\documentclass[11pt]{amsart}
\usepackage{graphicx}
\usepackage{tgpagella}
\usepackage{euler}
\usepackage[T1]{fontenc}
\usepackage{amssymb}
\usepackage{amsmath,amsthm}
\usepackage[hidelinks]{hyperref}
\usepackage{amsthm}
\usepackage[english]{babel}
\usepackage{mathrsfs}

\newtheorem{main}{Theorem}

\newtheorem{theorem}{Theorem}[section]
\newtheorem{lem}[theorem]{Lemma}
\newtheorem{prop}[theorem]{Proposition}

\newtheorem{cor}[theorem]{Corollary}

\theoremstyle{definition}
\newtheorem{definition}[theorem]{Definition}
\newtheorem{notation}[theorem]{Notation}
\newtheorem{Remark}[theorem]{Remark}

\def\ca{\curvearrowright}
\def\ra{\rightarrow}

\def\la{\lambda}
\def\La{\Lambda}

\def\bb{\mathbb}

\def\bb{\mathbb}
\def\g{\gamma}

\def\G{\Gamma}
\def\Cal{\mathcal}

\numberwithin{equation}{section}

\def\mod{\rm Mod}

\begin{document}

\title[$W^*$-superrigidity for braid group actions]{$W^*$-superrigidity for arbitrary actions  of central quotients of braid groups}
\author[I. Chifan]{Ionut Chifan}
\address{Department of Mathematics, University of Iowa, 14 MacLean Hall, IA
52242, USA}
\email{ionut-chifan@uiowa.edu}
\thanks{I.C.\ was supported in part by NSF Grant \#1301370. }
\thanks{A.I.\ was supported in part by  NSF  Grant DMS \#1161047 and NSF Career Grant DMS \#1253402.}
\author[A. Ioana]{Adrian Ioana}
\address{Department of Mathematics, University of California San Diego,
La Jolla, CA 92093, USA}
\email{aioana@math.ucsd.edu}
\author[Y. Kida]{Yoshikata Kida}
\address{Department of Mathematics, Kyoto University, 606-8502 Kyoto, Japan}
\email{kida@math.kyoto-u.ac.jp}
\dedicatory{}

\begin{abstract}
For any $n\geqslant 4$ let $\tilde B_n=B_n/Z(B_n)$  be the quotient of the braid group $B_n$  through its center. We prove that  any free ergodic probability measure preserving (pmp) action $\tilde B_n\curvearrowright (X,\mu)$ is W$^*$-superrigid in the following sense: if $L^{\infty}(X)\rtimes\tilde B_n\cong L^{\infty}(Y)\rtimes\Lambda$, for an arbitrary free ergodic pmp action $\Lambda\curvearrowright (Y,\nu)$, then the actions $\tilde B_n\curvearrowright X,\Lambda\curvearrowright Y$ are stably (or, virtually) conjugate.  Moreover, we prove that the same  holds if $\tilde B_n$ is replaced with a finite index subgroup of the direct product 
$\tilde B_{n_1}\times\cdots\times\tilde B_{n_k}$, for some $n_1,\ldots,n_k\geqslant 4$. The proof uses the dichotomy theorem for normalizers inside crossed products by free groups from \cite{PV11} in
 combination with the OE superrigidity theorem for actions of mapping class groups from \cite{Ki06}. Similar techniques allow us to prove that if a group $\Gamma$  is hyperbolic relative to a finite family of proper, finitely generated, residually finite, infinite subgroups, then the II$_1$ factor $L^{\infty}(X)\rtimes\Gamma$ has a unique Cartan subalgebra, up to unitary conjugacy, for any free ergodic pmp action $\Gamma\curvearrowright (X,\mu)$.
\end{abstract}

\maketitle

\section{Introduction}

\subsection{Background} Every  measure preserving  action $\Gamma\curvearrowright (X,\mu)$ of a countable group $\Gamma$ on a standard probability space $(X,\mu)$  gives rise to a tracial von Neumann algebra $M:=L^{\infty}(X)\rtimes\Gamma$, via the classical {\it group measure space construction} of Murray and von Neumann \cite{MvN36}. If the action $\Gamma\curvearrowright (X,\mu)$ is free and ergodic, then $M$ is a II$_1$ factor and $A:=L^{\infty}(X)$ is a {\it Cartan subalgebra}, i.e. a maximal abelian $*$-subalgebra whose {\it normalizer}, $\mathcal N_{M}(A)=\{u\in\mathcal U(M)|uAu^*=A\}$, generates a  dense subalgebra of $M$.

 Since the early 2000s, Popa's deformation/rigidity theory has generated spectacular progress in the classification of group measure space II$_1$ factors 
 (see the surveys \cite{Po06,Va10a,Io12b}). A major achievement of this theory is the recent discovery of the first families of actions that can be entirely reconstructed from their group measure space II$_1$ factors. 
To be more precise,  recall that a free ergodic pmp action $\Gamma\curvearrowright (X,\mu)$ is said to be   {\it W$^*$-superrigid} if any free ergodic pmp action $\Lambda\curvearrowright (Y,\nu)$ which gives rise to an isomorphic II$_1$ factor, $L^{\infty}(X)\rtimes\Gamma\cong L^{\infty}(Y)\rtimes\Lambda$, is {\it conjugate} (or, more generally, {\it stably conjugate}) to $\Gamma\curvearrowright (X,\mu)$.
 
 In \cite{Pe09}, J. Peterson proved the existence of W$^*$-superrigid actions. Shortly after, S. Popa and S. Vaes obtained the first concrete families of W$^*$-superrigid actions \cite{PV09}. For instance, they proved that for any group $\Gamma$ belonging to a large class of amalgamated free product groups,  the Bernoulli action $\Gamma\curvearrowright (X_0,\mu_0)^{\Gamma}$ is W$^*$-superrigid. The second named author then showed that the Bernoulli action $\Gamma\curvearrowright (X_0,\mu_0)^{\Gamma}$ is W$^*$-superrigid, for any icc property (T) group $\Gamma$ \cite{Io10}. Subsequently, several large families of W$^*$-superrigid actions were found in \cite{IPV10,CP10,HPV10,Va10b,CS11,PV11, Bo12}. 

 In particular, C. Houdayer, S. Popa and S. Vaes exhibited the first and, thus far, only  example of a group $\Gamma$ whose {\it every} free ergodic pmp action is W$^*$-superrigid \cite{HPV10}.
More precisely, let  $\Gamma=SL_3(\mathbb Z)*_{\Sigma}SL_3(\mathbb Z)$, where $\Sigma<SL_3(\mathbb Z)$ is the subgroup consisting of matrices $g=(g_{ij})$ with $g_{31}=g_{32}=0$. It was then shown in \cite{HPV10} that if  $\Gamma\curvearrowright (X,\mu)$ is an arbitrary free ergodic pmp action, then the II$_1$ factor $L^{\infty}(X)\rtimes\Gamma$ has a unique group measure space Cartan subalgebra, up to unitary conjugacy.  In combination with the third named author's work \cite{Ki09} showing that the action $\Gamma\curvearrowright (X,\mu)$ is {\it OE superrigid}, this implies that the action $\Gamma\curvearrowright (X,\mu)$ is W$^*$-superrigid.

\subsection{Statement of main results} The goal of this paper is to provide new and more natural examples of groups $\Gamma$ whose every action is W$^*$-superrigid.
To state our main results, let  $P_n<B_n$ be the pure braid and braid groups on $n$ strands. Denote by $Z$ the common center of $B_n$ and $P_n$, and define the central quotients $\tilde B_n=B_n/Z$ and $\tilde P_n=P_n/Z$. We show that if $n\geqslant 4$, then every free ergodic pmp action of $\tilde B_n$ and $\tilde P_n$ is W$^*$-superrigid. More generally, we have:

\begin{main}\label{A} 
Let  $\Gamma_1,\ldots,\Gamma_k$ be groups belonging to the family $\{\tilde B_n|n\geqslant 4\}$, for some $k\geqslant 1$. Let $\Gamma<\Gamma_1\times\cdots\times\Gamma_k$ be a finite index subgroup
and let $\Gamma\curvearrowright (X,\mu)$ be a free ergodic pmp action.
 Let $\Lambda\curvearrowright (Y,\nu)$ be an arbitrary free ergodic pmp action.

If $L^{\infty}(X)\rtimes\Gamma\cong L^{\infty}(Y)\rtimes\Lambda$, then  the actions $\Gamma\curvearrowright X$, $\Lambda\curvearrowright Y$ are stably conjugate. 
Moreover, if the action $\Gamma\curvearrowright X$ is aperiodic, then  the actions $\Gamma\curvearrowright X$, $\Lambda\curvearrowright Y$ are conjugate:   there exist  an isomorphism of probability spaces  $\theta:(X,\mu)\rightarrow (Y,\nu)$ and an isomorphism of groups $\delta:\Gamma\rightarrow\Lambda$  such that $\theta(\g\cdot x)=\delta(\g)\cdot\theta(x)$, for all $\g\in\Gamma$ and almost every $x\in X$.
\end{main}

For the notion of {\it stable conjugacy} (also called {\it virtual conjugacy}) of pmp actions, see Definition \ref{conj}.
An ergodic pmp action $\Gamma\curvearrowright (X,\mu)$ is called {\it aperiodic} if its restriction to any finite index subgroup $\Gamma_0<\Gamma$ is ergodic.

Note that since $P_n<B_n$  and hence $\tilde P_n<\tilde B_n$ has finite index, Theorem \ref{A} indeed applies to arbitrary actions of $\tilde B_n$ and $\tilde P_n$, for every $n\geqslant 4$.
\vskip 0.05in
In order to prove Theorem \ref{A}, we first show that the II$_1$ factor $L^{\infty}(X)\rtimes\Gamma$ has a unique Cartan subalgebra, up to unitary conjugacy.

\begin{main}\label{B}
\label{cartan}
Let  $\Gamma_1,\ldots,\Gamma_k$ be groups belonging to the family $\{\tilde B_n|n\geqslant 3\}$, for some $k\geqslant 1$. Let $\Gamma<\Gamma_1\times\cdots\times\Gamma_k$ be a finite index subgroup
and let $\Gamma\curvearrowright (X,\mu)$ be a free ergodic pmp action.
 Denote $M=L^{\infty}(X)\rtimes\Gamma$.

 If $A\subset M$ is a Cartan subalgebra, then there is a unitary $u\in M$ such that $A=uL^{\infty}(X)u^*$. 
\end{main}

For a discussion of how Theorem \ref{B} is proven and how it implies Theorem \ref{A}, see the next subsection.

For now, recall that a group $\Gamma$ is said to be $\mathcal C$-rigid ({\it Cartan}-rigid) in the sense of \cite{PV11} if the II$_1$ factor $L^{\infty}(X)\rtimes\Gamma$ has a unique Cartan subalgebra, up to unitary conjugacy, for any free ergodic pmp action $\Gamma\curvearrowright (X,\mu)$.
Theorem \ref{B} provides new examples of $\mathcal C$-rigid groups, adding to the classes already discovered in \cite{PV11,pv2,Io12a}.

In fact, a strategy similar to the one used in the proof of Theorem \ref{B} enabled us to find a new natural class of $\mathcal C$-rigid groups. 

\begin{main}\label{C}
\label{cartan}
Let  $\Gamma$ be a group  which is hyperbolic relative to a finite family of proper, finitely generated, residually finite, infinite subgroups. Let $\Gamma\curvearrowright (X,\mu)$ be a free ergodic pmp action and denote $M=L^{\infty}(X)\rtimes\Gamma$.

If $A\subset M$ is a Cartan subalgebra, then there is a unitary $u\in M$ such that $A=uL^{\infty}(X)u^*$. 
\end{main}

Here, we use the notion of relative hyperbolicity from \cite[Definition 1.6]{Os04}.

S. Popa and S. Vaes proved any non-elementary hyperbolic group is $\mathcal C$-rigid \cite{pv2}. In other words, any group that is relatively hyperbolic to the trivial subgroup (equivalently, to a family of finite subgroups), is $\mathcal C$-rigid.
In view of this, it is natural to wonder whether arbitrary relatively hyperbolic groups are $\mathcal C$-rigid.
Theorem \ref{C} provides strong supporting evidence towards this conjecture. 

The proof of Theorem \ref{C} relies on a combination of results from geometry group theory and deformation/rigidity theory. 
Thus, we first use powerful results  of D. Osin, F. Dahmani,  and V. Guirardel  on the structure of relatively hyperbolic groups \cite{Os05, DGO11}. These results are then combined with recent work of S. Popa, S. Vaes, and the second named author on the structure of normalizers inside crossed products by hyperbolic groups \cite{pv2} and by free product groups \cite{Io12a, Va13}.

 Theorem \ref{C} covers large families of groups extensively studied in various areas of mathematics such as topology, geometric group theory, or model theory. For instance, Theorem \ref{C} applies to the following classes of groups: 1) the fundamental group of any complete, finite-volume Riemannian manifold with pinched negative sectional curvature, \cite{Bow99,Fa99}; 2) any limit group, in the sense of Sela, over torsion free, relative hyperbolic groups with free abelian parabolics \cite{KM99,Se00}.        

\subsection{Comments on the proofs of Theorems \ref{A} and \ref{B}} 
Before outlining the proof of Theorem \ref{B}, let us explain how Theorem \ref{B} can be used to deduce Theorem \ref{A}.

Theorem \ref{B}  implies that if $\Gamma\curvearrowright (X,\mu)$ is as in Theorem \ref{A}, then the II$_1$ factor $L^{\infty}(X)\rtimes\Gamma$ has a unique Cartan subalgebra. 
Thus, any free ergodic pmp action 
$\Lambda\curvearrowright (Y,\nu)$ such that $L^{\infty}(X)\rtimes\Gamma\cong L^{\infty}(Y)\rtimes\Lambda$ is orbit equivalent to $\Gamma\curvearrowright (X,\mu)$. 
Now, as is well-known, $\tilde B_n$ is isomorphic to a finite index subgroup of the mapping class group of the sphere with $n+1$ punctures. Since $\Gamma$ is a finite index subgroup
of $\tilde B_{n_1}\times\cdots\times\tilde B_{n_k}$, for some $n_1,\ldots,n_k\geqslant 4$, then the third author's OE superrigidity theorem \cite{Ki06} gives that any free pmp action $\Lambda\curvearrowright (Y,\nu)$ that is orbit equivalent to $\Gamma\curvearrowright (X,\mu)$ is necessarily stably conjugate to it. 

Let us now outline the proof of Theorem \ref{B} in the case when $\Gamma=\tilde P_n$, for some $n\geqslant 3$.
The proof  has two main ingredients. The first one is S. Popa and S. Vaes' dichotomy theorem for normalizers \cite{PV11} which  asserts that if $A$ is an amenable  subalgebra of a crossed product $M=N\rtimes\mathbb F_n$, then either a corner of $A$ embeds into $N$, or the normalizing algebra $\mathcal N_{M}(A)''$ is amenable relative to $N$. In order to apply this result to our context, we use the fact that $P_n$ is an iterated semi-direct product of free groups.

More precisely, the second ingredient of the proof is the well-known fact that there is an exact sequence $1\rightarrow\mathbb F_{n-1}\rightarrow\tilde P_n\rightarrow\tilde P_{n-1}\rightarrow 1$, for every $n\geqslant 3$. Let $\mathcal F_n$ be the class of countable groups $\Gamma$ for which we can find a sequence of group homomorphisms $\Gamma=\Gamma_n\xrightarrow{\pi_n}\Gamma_{n-1}\xrightarrow{\pi_{n-1}}\cdots \xrightarrow{\pi_2}\Gamma_1$ such that $\Gamma_1$ and $\ker{\pi_i}$ ($2\leqslant i\leqslant n$) are non-abelian free groups. Since $\tilde P_3\cong\mathbb F_2$, the above fact gives that $\tilde P_n\in\mathcal F_{n-2}$, for every $n\geqslant 3$.

By using \cite{PV11} and an inductive argument,  for every $n\geqslant 1$, any group $\Gamma\in\mathcal F_n$ and every free ergodic pmp action $\Gamma\curvearrowright (X,\mu)$, we prove that if $A\subset M:=L^{\infty}(X)\rtimes\Gamma$  is a maximal abelian $*$-subalgebra whose normalizing algebra $\mathcal N_{M}(A)''$ has finite index in $M$, then a corner of $A$ embeds into $L^{\infty}(X)$. This implies that $M$ has a unique Cartan subalgebra, up to unitary conjugacy, and settles 
 Theorem \ref{B} in the case $\Gamma=\tilde P_n$, for $n\geqslant 3$.

\subsection{Acknowledgements} We would like to thank Remi Boutonnet and Cyril Houdayer for useful comments. The first author is especially  grateful to Denis Osin for kindly pointing out to him that the main results of this paper also apply to the families of groups described in the Section \ref{CC}.

\section{Preliminaries}

In this section we collect some notions and results that we will use in the proofs of our main results.

\subsection{Tracial von Neumann algebras} 
A {\it tracial von Neumann algebra} $(M,\tau)$ is a pair that consists of a von Neumann algebra $M$ and a faithful normal tracial state $\tau$. We denote by $M_{+}$ the set of all positive elements $x\in M$ and by $\mathcal Z(M)$ the {\it center} of $M$.
For $x\in M$, we denote by $\|x\|$ the operator norm of $x$, and by $\|x\|_2=\sqrt{\tau(x^*x)}$ the 2-$norm$ of $x$.
Throughout the paper, we denote by $L^2(M)$ the Hilbert space obtained by completing $M$ with respect to $\|\cdot\|_2$, and consider the standard representation $M\subset\mathbb B(L^2(M))$.

Let $(M,\tau)$ be a tracial von Neumann algebra and $P\subset M$ a von Neumann subalgebra. {\it Jones's basic construction} $\langle M,e_P\rangle\subset \mathbb B(L^2(M))$ is the von Neumann algebra generated by $M$ and the orthogonal projection $e_P:L^2(M)\rightarrow L^2(P)$. It is endowed with a faithful semifinite trace $Tr$ given by $Tr(xe_Py)=\tau(xy)$, for all $x,y\in M$. Also, we note that $E_P:={e_P}_{|M}:M\rightarrow P$ is the unique $\tau$-preserving conditional expectation onto $P$.

We say that $P\subset M$ is a {\it masa} if it is a maximal abelian $*$-subalgebra. 
The {\it normalizer of $P$ inside $M$}, denoted by $\mathcal N_{M}(P)$, is the set of all unitaries $u\in M$ such that $uPu^*=P$. We say that $P$ is {\it regular} in $M$ if $\mathcal N_{M}(P)''=M$. A {\it Cartan subalgebra} is a regular masa $P\subset M$.

If $\Cal M$ is a von Neumann algebra together with a subset $S\subset\Cal M$, then a state $\phi:\Cal M\rightarrow\mathbb C$ is called $S$-{\it central} if $\phi(xT)=\phi(Tx)$, for all $T\in\Cal M$ and $x\in S$. A tracial von Neumann algebra $(M,\tau)$ is called {\it amenable} if there exists an  $M$-central state $\phi:\mathbb B(L^2(M))\rightarrow\mathbb C$ such that $\phi(x)=\tau(x)$, for all $x\in M$. By a well-known theorem of A. Connes \cite{Co75}, $(M,\tau)$ is amenable if and only if it is approximately finite dimensional.

\subsection{Popa's intertwining-by-bimodules technique} We next recall S. Popa's powerful \\ {\it intertwining-by-bimodules} technique for conjugating subalgebras of tracial von Neumann algebras
 (see \cite [Theorem 2.1 and Corollary 2.3]{Po03}).

\begin {theorem}[Popa, \cite{Po03}]\label{corner} Let $(M,\tau)$ be a separable tracial von Neumann algebra and $P,Q$ be two (not necessarily unital) von Neumann subalgebras of $M$. 

Then the following are equivalent:

\begin{enumerate}

\item There exist  non-zero projections $p\in P, q\in Q$, a $*$-homomorphism $\theta:pPp\rightarrow qQq$  and a non-zero partial isometry $v\in qMp$ such that $\theta(x)v=vx$, for all $x\in pPp$.

\item There is no sequence $u_n\in\mathcal U(P)$ satisfying $\|E_Q(xu_ny)\|_2\rightarrow 0$, for all $x,y\in M$.
\end{enumerate}

\end{theorem}
If one of the two equivalent conditions from Theorem \ref{corner} holds then we say that {\it a corner of $P$ embeds into $Q$ inside $M$}, and write $P\prec_{M}Q$. If we moreover have that $Pp'\prec_{M}Q$, for any non-zero projection  $p'\in P'\cap 1_PM1_P$, then we write $P\prec_{M}^{s}Q$.

\subsection{Finite index inclusions of tracial von Neumann algebras} If $P\subset M$ are II$_1$ factors, then the {\it Jones index} of the inclusion $P\subset M$, denoted  $[M:P]$, is  the dimension of $L^2(M)$ as a left  $P$-module. M. Pimsner and S. Popa proved that the index $[M:P]$ is equal to the best constant appearing in several inequalities involving the conditional 
expectation $E_P$ \cite[Theorem 2.2]{PP86}. They also pointed out that these constants can be used to define the index of any inclusion of tracial von Neumann algebras \cite[Remark 2.4]{PP86}.  

\begin{definition}[Pimsner $\&$ Popa, \cite{PP86}] \label{index}
Let $(M,\tau)$ be a tracial von Neumann algebra  with a von Neumann subalgebra $P$. Let $$\lambda=\inf\;\{\|E_P(x)\|_2^2/\|x\|_2^2\;|\; x\in M_{+}, \; x\not=0\}.$$
The {\it index of the inclusion $P\subset M$} is defined by letting $[M:P]=\lambda^{-1}$, where we use the convention that $\frac{1}{0}=\infty$.
\end{definition}

Note that  if $M$ and $P$ are II$_1$ factors, then by \cite[Theorem 2.2]{PP86}, the usual index $[M:P]$ is equal to the one introduced in Definition \ref{index}. Also, notice that if $P\subset M$ has finite index, and $p\in P$ is a projection, then $pPp\subset pMp$ also has finite index.
This notion of  index is suitable for our purposes, as it allows us to prove the following technical result.

\begin{lem}\label{findex}
Let $(M,\tau)$ be a tracial von Neumann algebra, $p\in M$  a non-zero projection, and $P\subset pMp$  a von Neumann subalgebra. Let $N\subset M$ be a von Neumann subalgebra. Denote by $q\in N$ the support projection of $E_N(p)$, and by $P_0\subset qNq$ the unital von Neumann subalgebra generated by $E_N(P)$.  For every $t\geqslant 0$, we define $q_t=1_{[t,\infty)}(E_N(p))$, using  functional calculus.  

If the inclusion $P\subset pMp$ has finite index, then the inclusion $q_tP_0q_t\subset q_tNq_t$ has finite index, for every $t>0$.  
\end{lem}

\begin{proof} The proof we give follows closely the proof of \cite[Lemma 1.6 (1)]{Io11}. As the inclusion $P\subset pMp$ has finite index, there exists $\lambda>0$ such that $\|E_P(x)\|_2^2\geqslant \lambda\|x\|_2^2$, for all $x\in (pMp)_{+}$.

Let $t>0$ and notice that $E_N(p)q_t\geqslant tq_t$. Towards showing that the inclusion $q_tP_0q_t\subset q_tNq_t$ has finite index, take $x\in (q_tNq_t)_{+}$.

Firstly, since $x\in N$ and $E_N(E_P(pxp))\in P_0$ (by the definition of $P_0$) we get that 

\begin{equation}\label{ecu1}
|\tau(E_P(pxp)x)|=|\tau(E_N(E_P(pxp))x)|=|\tau(E_N(E_P(pxp))E_{P_0}(x))|\leqslant \|x\|_2\|E_{P_0}(x)\|_2.
\end{equation}

Secondly, since $pxp\in (pMp)_{+}$, we have that $\|E_P(pxp)\|_2^2\geqslant \lambda\|pxp\|_2^2$ and thus
\begin{equation}\label{ecu2}
|\tau(E_P(pxp)x)|=|\tau(E_P(pxp)pxp)|=\|E_P(pxp)\|_2^2\geqslant \lambda\|pxp\|_2^2.
\end{equation}

 Thirdly, since $E_N(p)q_t\geqslant tq_t$, we get that 
 \begin{equation}\label{ecu3}
 \|pxp\|_2^2=\tau(pxpx)=\|x^{\frac{1}{2}}px^{\frac{1}{2}}\|_2^2=\|x^{\frac{1}{2}}pq_tx^{\frac{1}{2}}\|_2^2\geqslant \|E_N(x^{\frac{1}{2}}pq_tx^{\frac{1}{2}})\|_2^2=
 \end{equation}
 $$  \|x^{\frac{1}{2}}E_N(p)q_tx^{\frac{1}{2}}\|_2\geqslant t^2\|x^{\frac{1}{2}}q_tx^{\frac{1}{2}}\|_2^2=t^2\|x\|_2^2.$$
 
 By combining equations \ref{ecu1}, \ref{ecu2} and \ref{ecu3} we conclude that $\|E_{P_0}(x)\|_2\geqslant \lambda t^2\|x\|_2$. Since $x\in (q_tMq_t)_{+}$ was arbitrary, the inclusion $q_tP_0q_t\subset q_tNq_t$ has finite index.
\end{proof}

\begin{lem} \cite[Lemma 2.3]{PP86}\label{ramen}
Let $(M,\tau)$ be a tracial von Neumann algebra and $P\subset M$ be a von Neumann subalgebra such that the inclusion $P\subset M$ has finite index. Then $M\prec_{M}^{s}P$.
\end{lem}

\begin{proof}
Assume by contradiction that $Mz\nprec_{M}P$, for some non-zero  projection $z\in \mathcal Z(M)$. Thus, $Mz\nprec_{Mz}Pz$, which by \cite[Theorem 2.1]{Po03} implies that $(Mz)'\cap\langle Mz,e_{Pz}\rangle$ does not contain any non-trivial projection of finite trace in $\langle Mz,e_{Pz}\rangle$.  Then \cite[Lemma 2.3]{PP86} implies that for every $\varepsilon>0$, we can find a projection $e\in Mz$ such that $\|E_{Pz}(e)\|_2<\varepsilon\|e\|_2$ (see \cite[Lemma 1.4]{Io11} for details). This contradicts the assumption that $P\subset M$ has finite index.
\end{proof}


\subsection{Dichotomy for normalizers inside crossed products by free groups}
Recently, S. Popa and S. Vaes proved a remarkable dichotomy for normalizers inside arbitrary crossed products $B\rtimes\Gamma$ by many groups $\Gamma$, including the free groups \cite[Theorem 1.6]{PV11}. 
In order to state their theorem, we first need to recall the notion of relative amenability.

\begin{definition} \cite[Section 2.2]{OP07}
Let $(M,\tau)$ be a tracial von Neumann algebra, $p\in M$ a projection, and $P\subset pMp,Q\subset M$ von Neumann subalgebras. We say that $P$ is {\it amenable relative to $Q$ inside $M$} if there exists a $P$-central state $\phi:p\langle M,e_Q\rangle p\rightarrow\mathbb C$ such that $\phi(x)=\tau(x)$, for all $x\in pMp$.
\end{definition}

\begin{theorem} [Popa $\&$ Vaes, \cite{PV11}]\label{pv} Let $\Gamma$ be a weakly amenable group that admits a proper $1$-cocycle into an orthogonal representation that is weakly contained in the regular representation. Let $\Gamma\curvearrowright B$ be a trace preserving action on a tracial von Neumann algebra $(B,\tau)$. Denote $M=B\rtimes\Gamma$. Let $p\in M$ be a projection and $A\subset pMp$ a von Neumann subalgebra that is amenable relative to $B$ inside $M$. 

Then either $A\prec_{M}B$ or $P:=\mathcal N_{pMp}(A)''$ is amenable relative to $B$ inside $M$.
\end{theorem}

Recall from \cite{OP07} that a II$_1$ factor $M$ is called {\it strongly solid} if $\mathcal N_{M}(A)''$ is amenable, for any diffuse amenable von Neumann subalgebra $A\subset M$.

\begin{definition} We say that a group $\Gamma$ is {\it relatively strongly solid} if it is non-amenable and satisfies the dichotomy of Theorem \ref{pv}. We denote by 
$\Cal C_{rss}$ the class of all such groups. \end{definition}

Note that  by Theorem \ref{pv}, $\mathcal C_{rss}$ contains in particular all non-abelian free groups.

\subsection{Braid and pure braid group}\label{braidgroups}
In this subsection we record several algebraic properties of the braid and pure braid groups (see for instance \cite[Chapter 9]{FM11}). Recall that the {\it braid group on $n$ strands}, denoted  by $B_n$, has the following presentation $$B_n=\langle\sigma_1,\ldots,\sigma_{n-1}|\;\sigma_i\sigma_{i+1}\sigma_i=\sigma_{i+1}\sigma_i\sigma_{i+1},\;\text{for all $i$},\;\; \text{and  $\sigma_i\sigma_j=\sigma_j\sigma_i,$\;\; if $|i-j|>1$}\rangle$$

The {\it pure braid group $P_n$} is the kernel of the natural homomorphism $B_n\rightarrow S_n$, where $S_n$ denotes the  group of permutations of $\{1,\ldots,n\}$.  In other words, $P_n$ is the subgroup of all braids that induce the trivial permutation on $\{1,\ldots,n\}$. The centers of $B_n$ and $P_n$ coincide, $Z(B_n)=Z(P_n)$, and  are 
isomorphic with the infinite cyclic group. Moreover, $P_n$ splits as a direct product over its center: $P_n\cong P_n/Z(P_n)\times Z(P_n)$.

 Most of the rigidity results that we will prove apply to the central quotients of $B_n$ and $P_n$, which we denote throughout by $\tilde B_n=B_n/Z(B_n)$ and $\tilde P_n=P_n/Z(P_n)$. It is clear that $\tilde P_n$ is isomorphic to a finite index subgroup of $\tilde B_n$. Also, since $P_3\cong \mathbb F_2\times \bb Z$, we have that $\tilde P_3\cong \bb F_2$. 

Next, we isolate the key property of $\tilde P_n$ and $\tilde B_n$ that will allow us to prove uniqueness of Cartan subalgebras for II$_1$ factors associated with their actions. First, we introduce some terminology.

\begin{definition}\label{chain}Let $\Cal C$ be a class of groups. We define $Quot_{1}(\mathcal C)=\mathcal C$. Given an integer $n\geqslant 2$, we say that a group $\Gamma$ belongs to the class $Quot_n(\Cal C)$ if there exist: 
\begin{enumerate}
\item a collection of groups  $\Gamma_k$, $1\leqslant k\leqslant n$, such that $\G_1\in \Cal C$, $\G$ is commensurable to $\G_n$ (i.e. there are finite index subgroups $\Lambda<\G$ and $\Sigma<\G_n$ with $\Lambda\cong\Sigma$), and
\item a collection of surjective homomorphisms $ \pi_k:\G_{k}\rightarrow \G_{k-1}$ such that $ker(\pi_k)\in \Cal C$, for all $2\leqslant k\leqslant n$.
\end{enumerate}
\end{definition}

\begin{lem}\label{quot}
Let $\Cal C$ be a class of groups and $\Gamma_k$, $1\leqslant k\leqslant n$, be a collection of groups, for $n\geqslant 2$. Assume that there exist surjective homomorphisms $ \pi_k:\G_{k}\rightarrow \G_{k-1}$ satisfying $ker(\pi_k)\in \Cal C$, for all $2\leqslant k\leqslant n$. If we denote by $p_n=\pi_2\circ\pi_3\circ\cdots\circ\pi_n:\Gamma_n\rightarrow\Gamma_1$ then $\ker(p_n)\in Quot_{n-1}(\Cal C).$
\end{lem}

\begin{proof}
For every $2\leqslant k\leqslant n$, denote by $p_k=\pi_2\circ\pi_3\circ\cdots\circ\pi_k:\Gamma_k\rightarrow\Gamma_1$ and by $\Lambda_k=\ker(p_k)<\Gamma_k$. Then $\pi_k(\Lambda_k)=\Lambda_{k-1}$ and hence the restriction ${\pi_k}_{|\Lambda_k}:\Lambda_k\rightarrow \Lambda_{k-1}$ is well-defined and surjective. Moreover, we have that $\Lambda_2=\ker(\pi_2)\in\Cal C$ and $\ker({\pi_k}_{|\Lambda_k})=\ker(\pi_k)\in\Cal C$, for every $3\leqslant k\leqslant n$. This shows that $\ker(p_n)=\Lambda_n\in Quot_{n-1}(\Cal C)$.
\end{proof}

We also record the following basic fact.

\begin{lem}
\label{prod}
Let $\Cal C$ be a class of groups. 
If $\Gamma_1\in Quot_{n_1}(\mathcal C),\ldots,\Gamma_k\in Quot_{n_k}(\mathcal C)$, then $\Gamma_1\times\ldots\times\Gamma_k\in Quot_{n_1+\cdots+n_k}(\mathcal C)$.

\end{lem}

Now, for every $n\geqslant 3$, there exists a surjective homomorphism $\pi_n:P_{n}\rightarrow P_{n-1}$. More precisely, given a pure braid $\gamma\in P_{n}$, we can remove its last strand and obtain a pure braid $\pi_n(\gamma)\in P_{n-1}$. Moreover, we have that $\pi_n(Z(P_n))=Z(P_{n-1})$, and the kernel of $\pi_n$ is isomorphic to the free group on $n-1$ generators, $\mathbb F_{n-1}$, by the Birman exact sequence.  Hence, $\tilde\pi_n:\tilde P_n\rightarrow\tilde P_{n-1}$ given by $\tilde\pi_n(\gamma Z(P_n))=\pi_n(\gamma)Z(P_{n-1})$ is a well-defined surjective homomorphism and $\ker{\tilde\pi_n}\cong\ker{\pi_n}\cong\mathbb F_{n-1}$. Since $\tilde P_n<\tilde B_n$ is a finite index subgroup and $\tilde P_3\cong\mathbb F_2$, we altogether deduce the following:

\begin{cor}\label{surjections} For every $n\geqslant 3$, we have that $\tilde B_n, \tilde P_n\in Quot_{n-2}(\Cal F)$, where $\Cal F$ denotes the class of all non-abelian free groups.  
\end{cor}

\subsection{OE superrigidity for actions of central quotients of braid groups} 
We begin by recalling some terminology.

\begin{definition}\label{conj}
Two ergodic pmp actions $\Gamma\curvearrowright (X,\mu)$ and $\Lambda\curvearrowright (Y,\nu)$ are said to be
\begin{enumerate}
 \item {\it conjugate} if there exist a probability space isomorphism $\theta:(X,\mu)\rightarrow (Y,\nu)$ and a group isomorphism $\delta:\Gamma\rightarrow\Lambda$ such that $\theta(\g\cdot x)=\delta(\g)\cdot\theta(x)$, for all $\g\in\Gamma$ and almost every $x\in X$.
 \item {\it stably conjugate} if there exist finite index subgroups $\Gamma_0<\Gamma$, $\Lambda_0<\Lambda$, and finite normal subgroups $\G_1\unlhd \Gamma_0,\Lambda_1\unlhd\Lambda_0$ such that
 \begin{itemize}
 \item the action $\Gamma\curvearrowright (X,\mu)$ is induced from some pmp action $\Gamma_0\curvearrowright (X_0,\mu_0)$.
 \item the action $\Lambda\curvearrowright (Y,\nu)$ is induced from some pmp action $\Lambda_0\curvearrowright (Y_0,\nu_0)$.
 \item the actions $\Gamma_0/\G_1\curvearrowright (X_0,\mu_0)/\G_1$ and $\Lambda_0/\Lambda_1\curvearrowright (Y_0,\nu_0)/\Lambda_1$ are conjugate.

 \end{itemize} 
 \item {\it orbit equivalent} if there exists a probability space isomorphism $\theta:(X,\mu)\rightarrow (Y,\nu)$ such that $\theta(\Gamma\cdot x)=\Lambda\cdot\theta(x)$, for almost every $x\in X$.
 \end{enumerate} 
\end{definition}

Here, we say that an action $\Gamma\curvearrowright (X,\mu)$ is {\it induced} from an action $\Gamma_0\curvearrowright (X_0,\mu_0)$ of a finite index subgroup $\Gamma_0<\Gamma$ if $X_0\subset X$ is a  $\Gamma_0$-invariant Borel subset of positive measure such that $\mu(\g\cdot X_0\cap X_0)=0$, for all $\g\in\Gamma\setminus\Gamma_0$.

\begin{Remark} The  notion of stable conjugacy is taken from \cite[Section 6.2]{PV09}. It is easy to see that two actions are stably conjugate if and only if they are {\it virtually conjugate}, in the sense of  \cite[Definition 1.3]{Ki06}.
\end{Remark}

Next, we state the third named author's OE superrigidity theorem for actions of mapping class group \cite[Theorem 1.1]{Ki06}.
Let $R_{g, n}$ be a compact orientable surface of genus $g$ with $n$ boundary components.
We denote by $\Gamma(R_{g, n})$ the group of isotopy classes of homeomorphisms of $R_{g, n}$, where isotopy may move points of the boundary of $R_{g, n}$. We also denote $\kappa(R_{g,n})=3g+n-4$.

\begin{theorem}[Kida, \cite{Ki06}]\label{yoshi} Let $k\geqslant 1$, and for $1\leqslant i\leqslant k$, let $R_{g_i, n_i}$ be a compact orientable surface such that $\kappa(R_{g_i,n_i})>0$.
Let $\Gamma$ be a finite index subgroup of $\Gamma(R_{g_1,n_1})\times\cdots\times\Gamma(R_{g_k,n_k})$. 
Let $\Gamma\curvearrowright (X,\mu)$ be a free ergodic pmp action. Let $\Lambda\curvearrowright (Y,\nu)$ be any free ergodic pmp action that is orbit equivalent to $\Gamma\curvearrowright (X,\mu)$.

Then the actions $\Gamma\curvearrowright (X,\mu)$ and $\Lambda\curvearrowright (Y,\nu)$ are stably conjugate. Moreover, if the action $\Gamma\curvearrowright (X,\mu)$ is  aperiodic, then the actions $\Gamma\curvearrowright (X,\mu)$ and $\Lambda\curvearrowright (Y,\nu)$ are conjugate.
\end{theorem}

\begin{Remark}\label{map}
Let $S_{g, n}$ be an orientable surface of genus $g$ with $n$ punctures.
We denote by $\mod(S_{g, n})$ the \textit{mapping class group} of $S_{g, n}$, i.e., the group of isotopy classes of orientation-preserving homeomorphisms of $S_{g, n}$.
For every $n\geq 3$, the group $\tilde B_n$ is isomorphic to an index $n$ subgroup of $\mod(S_{0, n+1})$ \cite[Section 9.2]{FM11}.

As precisely discussed in \cite[Section 5.1]{Iv02}, the groups $\mod$$(S_{g, n})$ and $\Gamma(R_{g, n})$ are isomorphic.
In conclusion, Theorem \ref{yoshi} applies to $\tilde B_n$ and $\tilde P_n$, and moreover to finite index subgroups of direct products of $\tilde B_n$'s and $\tilde P_n$'s, for every $n\geqslant 4$. 
\end{Remark}

\begin{Remark}\label{stable}
The proof of \cite[Theorem 1.1]{Ki06} moreover shows that any orbit equivalence $\theta:X\rightarrow Y$ between two actions  $\Gamma\curvearrowright (X,\mu)$ and $\Lambda\curvearrowright (Y,\nu)$ as in Theorem \ref{yoshi} arises in a canonical way from a stable conjugacy between  them.
\end{Remark}

\section{Uniqueness of Cartan subalgebras}
The main goal of this section is to prove the following theorem and its corollary:

\begin{theorem}\label{Cartan}Let $\G \in Quot_n(\Cal C_{rss})$, for some $n\geqslant 1$. Let $\G\curvearrowright (B,\tau)$ be a trace preserving action on a tracial  von Neumann algebra $(B,\tau)$. Denote $M=B\rtimes\Gamma$ and let $p\in M$ be a projection.
Let $A\subset pMp$ be a masa and denote $P=\mathcal N_{pMp}(A)''$. 

If  the inclusion $P\subset pMp$ has finite index, then $A\prec_{M}B$.
\end{theorem}

\begin{cor}\label{cartan}
Let $\G \in Quot_n(\Cal C_{rss})$, for some $n\geqslant 1$. Let $\Gamma\curvearrowright (X,\mu)$ be a free ergodic pmp action and denote $M=L^{\infty}(X)\rtimes\Gamma$.
If $A\subset M$ is a Cartan subalgebra, then we can find a unitary $u\in M$ such that $A=uL^{\infty}(X)u^*$.
\end{cor}

In preparation for the proof of Theorem \ref{Cartan}, we introduce some notation.

\begin{notation}\label{comu} In the setting from Theorem \ref{Cartan}, assume that we are given a countable group $\La$ and a homomophism $\delta:\G\ra \La$.  We denote by $\{u_\g\}_{\g\in\Gamma}\subset L\Gamma$ and $\{v_h\}_{h\in\Lambda}\subset L\Lambda$ the canonical unitaries.
Then we have a $*$-homomorphism $\theta:M\rightarrow M\bar{\otimes}L\Lambda$ given by 
 \begin{equation*}
\theta(bu_\g)=bu_\g\otimes v_{\delta(\g)},\;\;\;\text{for all $b\in B$ and $\g\in\Gamma$}.
\end{equation*} 
\end{notation}

Next, we establish two results about $\theta$ that we will use in the proof of Theorem \ref{cartan}.

\begin{prop}\label{prop1} In the setting from \ref{comu}, assume that $\delta$ is surjective. Let $P\subset M$ be a von Neumann subalgebra and $\Sigma<\Lambda$ a subgroup. 
If $\theta(P)\prec_{M\bar{\otimes}L\Lambda}M\bar{\otimes}L\Sigma$, then $P\prec_{M}B\rtimes\delta^{-1}(\Sigma)$.\end{prop} 

\begin{proof} Assume by contradiction that $P\nprec_{M}B\rtimes\delta^{-1}(\Sigma)$. Then by Theorem \ref{corner} we can find a sequence of unitaries $u_n\in P$ such that \begin{equation}\label{u_n}\|E_{B\rtimes\delta^{-1}(\Sigma)}(au_nb)\|_2\rightarrow 0,\;\;\;\text{for all $a,b\in M$}.\end{equation}

We claim that \begin{equation}\label{weakly} \|E_{M\bar{\otimes}L\Sigma}(x\theta(u_n)y)\|_2\rightarrow 0,\;\;\;\text{for all $x,y\in M\bar{\otimes}L\Lambda$}.\end{equation}

By using approximations in $\|\cdot\|_2$ it suffices to prove \ref{weakly} whenever $x=1\otimes v_{\la_1}$ and $y=1\otimes v_{\la_2}$, for some $\la_1,\la_2\in\Lambda$. Let $\g_1,\g_2\in\Gamma$ such that $\la_1=\delta(\g_1)$ and $\la_2=\delta(\g_2)$.

For every $n$, we decompose $u_n=\sum_{\g\in\Gamma}x_n^\g u_\g$, where $x_n^\g\in B$. Then we have that
$E_{M\bar{\otimes} L\Sigma}((1\otimes v_{\la_1})\theta(u_n)(1\otimes v_{\la_2}))=\sum_{\g\in\Gamma, \;\delta(\g_1\g\g_2)\in\Sigma}x_n^\g u_\g\otimes v_{\delta(\g_1\g\g_2)}. $

This further implies that \begin{equation}\label{E_M}\|E_{M\bar{\otimes} L\Sigma}((1\otimes v_{\la_1})\theta(u_n)(1\otimes v_{\la_2}))\|_2^2=\sum_{\g\in \g_1^{-1}\delta^{-1}(\Sigma)\g_2^{-1}}\|x_n^\g\|_2^2= \end{equation} $$\|E_{B\rtimes\delta^{-1}(\Sigma)}(u_{\g_1}u_nu_{\g_2})\|_2^2.$$

By combining \ref{u_n} and \ref{E_M}, we conclude that $\|E_{M\bar{\otimes} L\Sigma}((1\otimes v_{\la_1})\theta(u_n)(1\otimes v_{\la_2}))\|_2\rightarrow 0$, which proves claim \ref{weakly}. Following Theorem \ref{corner}, this contradicts the assumption that $\theta(P)\prec_{M\bar{\otimes}L\Lambda}M\bar{\otimes}L\Sigma$.
\end{proof}

\begin{prop}\label{prop2}
In the setting from \ref{comu}, let $p\in \theta(M)$ be a non-zero projection. If $p\theta(M)p$ is amenable relative to $M\otimes 1$ inside $M\bar{\otimes}L\Lambda$, then $\delta(\Gamma)$ is amenable. 
\end{prop} 

\begin{proof} Assume that $p\theta(M)p$ is amenable relative to $M\otimes 1$. Then we can find a non-zero projection $q\in \theta(M)'\cap (M\bar{\otimes}L\Lambda)$ such that $\theta(M)q$ is amenable relative to $M\otimes 1$ inside $M\bar{\otimes}L\Lambda$ (see e.g. \cite[Remark 2.2.]{Io12a}). Therefore, there exists a $\theta(M)q$-central state $\Psi:q\langle M\bar{\otimes}L\Lambda,e_{M\otimes 1}\rangle q\rightarrow\mathbb C$ such that ${\Psi}_{|q(M\bar{\otimes}L\Lambda)q}=\tau$.

Let $\pi: M\bar{\otimes}\mathbb B(\ell^2\Lambda)\rightarrow \langle M\bar{\otimes}L\Lambda,e_{M\otimes 1}\rangle$ be the obvious $*$-isomorphism.
We define $\Phi:\mathbb B(\ell^2\Lambda)\rightarrow\mathbb C$ by letting $\Phi(T)=\Psi(q\pi(1\otimes T)q)$. Fixing $T\in\mathbb B(\ell^2\Lambda)$ and $\g\in\Gamma$ we have that  $\pi(1\otimes v_{\delta(\g)}Tv_{\delta(\g)}^*)=(u_\g\otimes v_{\delta(\g)})\pi(1\otimes T)(u_\g\otimes v_{\delta(\g)})^*$. Since $(u_\g\otimes v_{\delta(\g)})q\in \theta(M)q$ and $\Psi$ is $\theta(M)q$-central we get that

 \begin{equation}\label{regular} \Phi(v_{\delta(\g)}Tv_{\delta(\g)}^*)=\Psi(((u_\g\otimes v_{\delta(\g)})q)\pi(1\otimes T)((u_\g\otimes v_{\delta(\g)})q)^*)=\Phi(T) \end{equation}
 
 Since $\Phi(1)=\tau(q)>0$ then the identity \ref{regular} implies that $\tau(q)^{-1}\Phi:\mathbb B(\ell^2\Lambda)\rightarrow\mathbb C$ is a $\{v_{\delta(\g)}\}_{\g\in\Gamma}$-central state. This clearly implies that $\delta(\Gamma)$ is amenable.
\end{proof} 

In the proof of Theorem \ref{Cartan} we will also need the following technical result:

\begin{prop}\label{masa}
Let $(M,\tau)$ be a tracial von Neumann algebra, $A\subset M$  a masa, and $N\subset M$ a von Neumann subalgebra. 
Denote $P=\mathcal N_{M}(A)''$. Assume that $P\subset M$ has finite index and that $A\prec_{M}N$.

Then we can find a non-zero projection $q\in N$ and a masa $A_0\subset qNq$ such that $P_0\subset qNq$ has finite index, where $P_0=\mathcal N_{qNq}(A_0)''$.  Moreover, we can find non-zero projections $p_0\in A$, $q'\in A_0'\cap qMq$, and a unitary $u\in M$ such that $u(Ap_0)u^*=A_0q'$.
\end{prop}

\begin{proof} Since $A\prec_{M}N$, by applying Theorem \ref{corner} we can find projections $p\in A,q\in N$, a $*$-homomorphism $\theta:Ap\rightarrow qNq$, and a non-zero partial isometry $v\in qMp$ such that $\theta(x)v=vx$, for all $x\in Ap$.  Since $A\subset M$ is a masa, we may assume that $v^*v=p$. Also, if  $q'=vv^*\in \theta(Ap)'\cap qMq$, then we can suppose that the support of $E_N(q')$ equals $q$.

Moreover, by  the proof of \cite[Theorem A.2]{Po01} and \cite[Lemma 1.5]{Io11}, we may assume that $A_1:=\theta(Ap)\subset qNq$ is a masa. Denote $P_1=\mathcal N_{qNq}(A_1)''$.

Since $A\subset P$ is regular, by \cite[Lemma 3.5 (2)]{Po03} it follows that $Ap\subset pPp$ is regular. 
Now, let $u\in\mathcal N_{pPp}(Ap)$ and define $\alpha_u\in$ Aut$(Ap)$ by  letting $\alpha_u(x)=uxu^*$.  Then we have that $\beta_u:=\theta\circ\alpha_u\circ\theta^{-1}\in$ Aut$(A_1)$ satisfies $\beta_u(x)(vuv^*)=(vuv^*)x$, for all $x\in A_1$.

By projecting onto $N$, we get that $\beta_u(x)E_N(vuv^*)=E_N(vuv^*)x$, for all $x\in A_1$.
Since $A_1\subset qNq$ is a masa and $E_N(vuv^*)\in qNq$, by \cite[Lemma 2.1]{JP81} we deduce that $E_N(vuv^*)\in P_1.$ Since $Ap\subset pPp$ is regular, we conclude that $E_N(vpPpv^*)\subset P_1$.

 For $t\geqslant 0$, let $q_t=1_{[t,\infty)}(E_N(q'))$. Then $q_t\in A_1'\cap qNq=A_1$, for all $t\geqslant 0$, and $\|q_t-q\|\rightarrow 0$, as $t\rightarrow 0$. In particular, we can find $t>0$ such that $q_t\not=0$.
Since $P\subset M$ has finite index,  $vpPpv^*\subset vpMpv^*=q'Mq'$ has finite index. By Lemma \ref{findex} we get that $q_tE_N(vpPpv^*)''q_t\subset q_tNq_t$ has finite index, hence $q_tP_1q_t\subset q_tNq_t$ has finite index.

Finally, since $A_1\subset P_1$ is regular, by \cite[Lemma 3.5.(2)]{Po03} we derive that $A_1q_t\subset q_tP_1q_t$ is regular. Thus, if  $A_0:=A_1q_t$, then the inclusion 
$P_0:=\mathcal N_{q_tNq_t}(A_0)''\subset q_tNq_t$ has finite index. Moreover, if $u\in M$ is any unitary extending $v$ and $p_0=\theta^{-1}(q_t)$, then $u(Ap_0)u^*=A_0(q_tq')$. Since $q_tq'\in (A_0q_t)'\cap q_tMq_t$, we are done.
\end{proof}

\subsection*{Proof of Theorem \ref{Cartan}}  We start the proof with two claims. Firstly, we claim that if a group $\Gamma$ satisfies  Theorem \ref{Cartan} then, in the setting from Theorem \ref{Cartan},  we moreover get that $A\prec_{M}^{s}B$. Indeed, by \cite[Lemma 2.5 and Proposition 2.6]{Va10b} we can find  projections $p_1,p_2\in \mathcal Z(P)$ such that $p_1+p_2=p$, $Ap_1\prec_M^{s}B$ and $Ap_2\nprec_{M}B$. Assume by contradiction that $p_2\not=0$. Then since $Pp_2\subset\mathcal N_{p_2Mp_2}(Ap_2)''$, $Pp_2\subset p_2Mp_2$ has finite index, and $\Gamma$ satisfies the conclusion of Theorem \ref{Cartan}, we would get that $Ap_2\prec_{M}B$.

Secondly, we claim that if $\Gamma_0<\Gamma$ is a finite index inclusion of groups and $\Gamma_0$ satisfies the conclusion of Theorem \ref{Cartan}, then $\Gamma$ also does. 
Let $M=B\rtimes\Gamma$ the crossed product algebra associated with a trace preserving action $\Gamma\curvearrowright (B,\tau)$.  Let $A\subset pMp$ be a masa such that $P=\mathcal N_{pMp}(A)''$ has finite index in $pMp$.
 Denote $M_0=B\rtimes\Gamma_0$. Since $\Gamma_0<\Gamma$ has finite index, we have that $M\prec_{M}M_0$. 
 By Proposition \ref{masa} we can find a non-zero projection $q\in M_0$ and  a masa $A_0\subset qM_0q$ such that  $P_0:=\mathcal N_{qM_0q}(A_0)''\subset qM_0q$ has finite index.
 Moreover, we can find non-zero projections  $p_0\in A$, $q'\in A_0'\cap qMq$, and  a unitary $u\in M$ such that $u(Ap_0)u^*=A_0q'$. Since $\Gamma_0$ satisfies the conclusion of Theorem \ref{Cartan}, our first claim yields that $A_0\prec_{M_0}^sB$. This implies that $A\prec_{M}B$, and shows that $\Gamma$ satisfies Theorem \ref{Cartan}, as claimed.

In the rest of the proof, we proceed by induction on $n$  to prove that every group $\Gamma\in\text{Quot}_n(\mathcal C_{rss})$ satisfies the conclusion of Theorem \ref{Cartan}.

First assume that $n=1$ and let $\Gamma\in Quot_1(\mathcal C_{rss})=\mathcal C_{rss}$. Let $M=B\rtimes\Gamma$ and $A\subset pMp$ be as in the hypothesis of Theorem \ref{Cartan}. Since $\Gamma\in\mathcal C_{rss}$ we have that either $A\prec_{M}B$ or $P=\mathcal N_{pMp}(A)''$ is amenable relative to $B$. 

Let us assume that $P$ is amenable relative to $B$ and derive a contradiction. Note that since $P\subset pMp$ has finite index, by Lemma \ref{ramen} we get that $pMp\prec_{pMp}^{s}P$. This easily implies that $pMp$ is amenable relative to $P$. Since $P$ is amenable relative to $B$, by using \cite[Proposition 2.4 (3)]{OP07} we get that $pMp$ is amenable relative to $B$. Therefore, $Mz$ is amenable relative to $B$, where $z\in M$ is the central support of $p$ (see \cite[Remark 2.2]{Io12a}). 

Thus, we can find a $Mz$-central state $\phi:z\langle M,e_B\rangle z\rightarrow\mathbb C$. Next, note that there is a $*$-isomorphism $\alpha:\langle M,e_B\rangle\rightarrow \mathbb B(\ell^2\Gamma)\bar{\otimes}B$ which satisfies $\alpha(u_\g)=\lambda_\g\otimes\sigma_\g$, where $\{\lambda_\g\}_{\g\in\Gamma}\subset\mathcal U(\ell^2\Gamma)$ is the regular representation and $\{\sigma_\g\}_{\g\in\Gamma}\subset\mathcal U(L^2(B))$ is the unitary representation associated with the action $\Gamma\curvearrowright (B,\tau)$. Let $\psi:\mathbb B(\ell^2\Gamma)\rightarrow\mathbb C$ be the state given by $\psi(T)=\phi(z(T\otimes 1)z)$. 

Then for all $T\in\mathbb B(\ell^2\Gamma)$ and $\g\in\Gamma$, we have $$\Psi(\lambda_\g T\lambda_\g^*)=\phi(z(\lambda_\g\otimes 1)(T\otimes 1)(\lambda_\g\otimes 1)^*z)=$$ $$\phi(z(\lambda_\g\otimes\sigma_\g)(T\otimes 1)(\lambda_\g\otimes\sigma_g)^*z)=\phi(zTz)=\psi(T).$$

Thus, $\psi$ is $\{\lambda_\g\}_{\g\in\Gamma}$-invariant and therefore $\Gamma$ is amenable, which is a contradiction. This shows that we must have $A\prec_{M}B$ and finishes the proof in the case $n=1$.

Let $n\geqslant 1$ and assume that the conclusion of Theorem \ref{Cartan} holds for any group in $Quot_n(\mathcal C_{rss})$. Let $\Gamma\in Quot_{n+1}(\mathcal C_{rss})$. By Definition \ref{chain} and Lemma \ref{quot}  we can find a group $\Gamma'$ that is commensurable to $\Gamma$,  a group $\Lambda\in\mathcal C_{rss}$, and a surjective homomorphism $\delta:\Gamma'\rightarrow\Lambda$ such that $\Sigma:=\ker{\delta}\in Quot_n(\mathcal C_{rss})$. Thus, we may assume that $\Gamma''=\Gamma\cap\Gamma'$ has finite index in both $\Gamma$ and $\Gamma'$. By our second claim, in order to show that $\Gamma$ satisfies Theorem \ref{Cartan}, it is enough to argue that $\Gamma''$ does. Moreover, it is easy to see that if $\Gamma'$ satisfies Theorem \ref{Cartan}, then $\Gamma''$ does.
Altogether,  we may assume that $\Gamma=\Gamma'$.

Let $M=B\rtimes\Gamma$ the crossed product algebra associated with a trace preserving action $\Gamma\curvearrowright (B,\tau)$.  Let $A\subset pMp$ be a masa such that $P=\mathcal N_{pMp}(A)''$ has finite index in $pMp$. Our goal is to show that $A\prec_{M}B$.

To this end, let $\theta:M\rightarrow M\bar{\otimes}L(\Lambda)$ be the $*$-homomorphism defined in Notation \ref{comu}. 
We identify $M\bar{\otimes}L(\Lambda)=M\rtimes\Lambda$, where $\Lambda$ acts trivially on $M$. Since $\Lambda\in\mathcal C_{rss}$ and $\theta(P)\subset\mathcal N_{\theta(p)(M\bar{\otimes}L(\Lambda))\theta(p)}(\theta(A))''$, we have that either $\theta(A)\prec_{M\bar{\otimes}L(\Lambda)}M\otimes 1$, or $\theta(P)$ is amenable relative to $M\otimes 1$ inside $M\bar{\otimes}L(\Lambda)$. 

Let us show that the second case leads to a contradiction. Assume that $\theta(P)$ is amenable relative to $M\otimes 1$ inside $M\bar{\otimes}L(\Lambda)$. Since $P\subset pMp$ has finite index, by Lemma \ref{ramen} we get that $pMp\prec_{pMp}^sP$, hence $pMp$ is amenable relative to $P$ inside $pMp$. From this we derive that $\theta(pMp)$ is amenable relative to $\theta(P)$ inside $M\bar{\otimes}L(\Lambda)$.  By applying \cite[Proposition 2.4 (3)]{OP07}  we get that $\theta(pMp)$ is amenable relative to $M\otimes 1$ inside $M\bar{\otimes}L(\Lambda)$. Since $\delta$ is surjective, Proposition \ref{prop2} implies that $\Lambda$ is amenable, which is a contradiction.
 
Therefore,  the first case must hold, i.e. $\theta(A)\prec_{M\bar{\otimes}L(\Lambda)}M\otimes 1$. Then Proposition \ref{prop1} implies that 
$A\prec_{M}N:=B\rtimes\Sigma$. By Lemma \ref{masa} we can find a non-zero projection $q\in N$ and a masa  $A_0\subset qNq$ such that 
$P_0:=\mathcal N_{qNq}(A_0)''\subset qNq$ has finite index. Moreover, we can find non-zero projections $p_0\in A$, $q'\in A_0'\cap qMq$, and a unitary $u\in M$ such that $u(Ap_0)u^*=A_0q'$. 

Since $\Sigma\in Quot_n(\mathcal C_{rss})$, the induction hypothesis gives that $\Sigma$ verifies conclusion of Theorem \ref{Cartan}. Therefore, by using the first claim of the proof, we deduce that $A_0\prec_{N}^sB$. This together with the equality $u(Ap_0)u^*=A_0q'$ implies that $A\prec_{M}B$, which finishes the proof of the theorem. \hfill$\square$

\subsection*{Proof of Corollary \ref{cartan}} Let $A\subset  M=L^{\infty}(X)\rtimes\Gamma$ be a Cartan subalgebra. Then by applying Theorem \ref{Cartan} we get that $A\prec_{M}L^{\infty}(X)$. 
By \cite[Theorem A.1]{Po01} we conclude that $A=uL^{\infty}(X)u^*$, for some unitary $u\in M$. \hfill$\square$

\section{Proofs of Theorems \ref{A} and \ref{B}}

\subsection{Proof of Theorem \ref{B}} Let $\Gamma<\Gamma_1\times\cdots\times\Gamma_k$ be a finite index subgroup, where $\Gamma_1,\ldots,\Gamma_k$ are groups in the family $\{B_n|n\geqslant 3\}$.
Corollary \ref{surjections} gives that $\tilde B_n\in Quot_{n-2}(\mathcal F)$, where $\mathcal F$ is the class of all non-abelian free groups, for all $n\geqslant 3$. 
Then Lemma \ref{prod} implies that $\Gamma\in Quot_N(\mathcal F)$, for some  $N\geqslant 1$. 
Since we have that $\mathcal F\subset\mathcal C_{rss}$ by Popa-Vaes' dichotomy Theorem \ref{pv},  Corollary \ref{cartan} yields the conclusion. \hfill$\square$

\subsection{Proof of Theorem \ref{A}} The conclusion follows by combining together Theorem \ref{B}, Theorem \ref{yoshi} and Remark \ref{map}. \hfill$\square$

\begin{Remark}
We note that by combining Theorem \ref{B}, Theorem \ref{yoshi},  Remark \ref{map} and Remark \ref{stable}, it follows any action $\Gamma\curvearrowright (X,\mu)$ as in Theorem \ref{A} is {\it stably W$^*$-superrigid}, in the sense of \cite[Definition 6.4]{PV09}.
\end{Remark}

\vskip 0.1in

Theorem \ref{B} in combination with a result from \cite{Ki06} leads to a new class of groups whose every free ergodic action gives rise to a II$_1$ factor with trivial fundamental group.

\begin{cor}
Let $\Gamma$ be a group as in Theorem \ref{A} and $\Gamma\curvearrowright (X,\mu)$ be a free ergodic pmp action.

Then the II$_1$ factor $M=L^{\infty}(X)\rtimes\Gamma$ has trivial fundamental group, $\mathcal F(M)=\{1\}$.
\end{cor}

\begin{proof}
Since $L^{\infty}(X)$ is the unique Cartan subalgebra of $M$ by Theorem \ref{A}, we get that $\mathcal F(M)=\mathcal F(\mathcal R)$, where $\mathcal R$ is the equivalence relation $\mathcal R=\{(x,y)\in X\times X|\Gamma\cdot x=\Gamma\cdot y\}$. Since $\mathcal F(\mathcal R)=\{1\}$ by \cite[Corollary 2.8]{Ki06}, we are done.
\end{proof}

\vskip 0.1in
Since $P_n\cong\tilde P_n\times\mathbb Z$ and  $P_n<B_n$ has finite index,
Theorems \ref{A} and \ref{B} do not hold when $\Gamma$ is equal to either $B_n$ or $P_n$. 
However,  as a consequence of Theorem \ref{Cartan}, one still obtains some information about the structure of Cartan subalgebras in II$_1$ factors associated to actions of $B_n$ and $P_n$.

\begin{cor}
Let $\Gamma\in\{B_n,P_n\}$, for some $n\geqslant 3$, and  $Z$ be the common center of $B_n$ and $P_n$. 
 Let $\Gamma\curvearrowright (X,\mu)$ be a free ergodic pmp action and denote $M=L^{\infty}(X)\rtimes\Gamma$.
 
 If $A\subset M$ is a Cartan subalgebra, then $A\prec_{M}L^{\infty}(X)\rtimes Z$. 
\end{cor}

\begin{proof} Let $\Gamma_0=P_n$, viewed as a subgroup of $\Gamma$. Recall that if $Z$ denotes the center of $P_n$ and $\tilde P_n=P_n/Z$, then $P_n\cong\tilde P_n\times Z$. Denote $M_0=L^{\infty}(X)\rtimes\Gamma_0$.

Since $\Gamma_0<\Gamma$ has finite index, by Proposition \ref{masa}  we can find a non-zero projection $q\in M_0$ and a masa 
$A_0\subset qM_0q$ such that  $P_0=\mathcal N_{qM_0q}(A_0)''\subset qM_0q$ has finite index.  
Moreover, we can find non-zero projections $p_0\in A$, $q'\in A_0'\cap qMq$, and a unitary $u\in M$ such that $u(Ap_0)u^*=A_0q'$.

Finally, we identify $M_0=(L^{\infty}(X)\rtimes Z)\rtimes\tilde P_n$. Since Corollary \ref{surjections} and Theorem \ref{pv} give that $\tilde P_n\in Quot_{n-2}(\mathcal F)$, by Theorem \ref{Cartan} we deduce that $A_0\prec_{M_0}L^{\infty}(X)\rtimes Z$. This easily implies that $A\prec_{M}L^{\infty}(X)\rtimes Z$.
\end{proof}

\subsection{W$^*$-superrigidity for actions of mapping class groups associated with surfaces of low genus}

In this subsection, we present a generalization of Theorem \ref{A} and explain how it provides evidence towards a general conjecture.

As mentioned before, the central quotient of the braid group $\tilde B_n$ can be canonically identified with a finite index subgroup of the mapping class group $\mod(S_{0,n+1})$ of an orientable surface $S_{0,n+1}$ of genus zero (sphere) with $n+1$ punctures. Thus, for any $n\geqslant 4$, we have $\mod(S_{0,n})\in Quot_{n-3}(\mathcal F)$ and our main unique Cartan subalgebra results apply to these groups as well. 

Next, we briefly argue that the situation is similar for the mapping class group $\mod(S_{g,n})$ of an orientable surface $S_{g,n}$ of genus $g=1,2$ with $n$ punctures. To see this, we need to recall some basic properties of these groups. Firstly, the results from \cite[page 57]{FM11} show that the mapping class group $\mod(S_{1,1})$ of the once-punctured two-torus $S_{1,1}$ can be identified with $SL(2,\mathbb Z)$, which is well known to contain a non-abelian free group of index $12$.
Also, by a result of Birman and Hilden from \cite{BiHi73}, the mapping class group $\mod(S_{2,0})$ is a $\mathbb Z_2$ central extension of the mapping class group $\mod(S_{0,6})$. Secondly, we recall that for all integers $g,k\geqslant 1$ we 
have Birman's short exact sequence of groups $1\ra \pi_1(S_{g,k})\ra \mod (S_{g,k+1})\ra \mod(S_{g,k}) \ra 1$. Here, $\pi_1(S_{g,k})$ denotes the fundamental group of $S_{g,k}$, which one can easily see that is isomorphic to the free group
of rank $2g+k-1$. Altogether, by arguing as in the subsection \ref{braidgroups}, these remarks imply that 
$\mod(S_{1,n})\in Quot_{n}(\mathcal F)$, for all $n\geqslant 1$, and $\mod(S_{2,n})\in Quot_{n+3}(\mathcal F)$, for all  $n\geqslant 0$. 

In conclusion, combining Corollary \ref{cartan} above with the main results in \cite{Ki06}, we obtain the following $W^*$-superrigidity statement:

\begin{cor}\label{mcglowgen} 
Fix an integer $k\geqslant 1$ and asssume that $\Gamma_1,\ldots,\Gamma_k$ are groups belonging to the family $\{\mod(S_{g,n})|g=0,1,2; \kappa(S_{g,n})>0; (g,n)\neq(2,0),(1,2)\}$. Let $\Gamma<\Gamma_1\times\cdots\times\Gamma_k$ be a finite index subgroup
and let $\Gamma\curvearrowright (X,\mu)$ be a free ergodic pmp action.
 Let $\Lambda\curvearrowright (Y,\nu)$ be an arbitrary free ergodic pmp action.

If $L^{\infty}(X)\rtimes\Gamma\cong L^{\infty}(Y)\rtimes\Lambda$, then  the actions $\Gamma\curvearrowright X$, $\Lambda\curvearrowright Y$ are stably conjugate. 
Moreover, if the action $\Gamma\curvearrowright X$ is aperiodic, then  the actions $\Gamma\curvearrowright X$, $\Lambda\curvearrowright Y$ are conjugate:   there exist  an isomorphism of probability spaces  $\theta:(X,\mu)\rightarrow (Y,\nu)$ and an isomorphism of groups $\delta:\Gamma\rightarrow\Lambda$  such that $\theta(\g\cdot x)=\delta(\g)\cdot\theta(x)$, for all $\g\in\Gamma$ and almost every $x\in X$.
\end{cor}

Note that the restrictions $\kappa(S_{g,n})>0$ and $(g,n)\neq(2,0),(1,2)$ in the statement above are necessary conditions carried over from the OE-superrigidity results in \cite{Ki06}. 

In view of this, we conjecture that Corollary \ref{mcglowgen} holds  for any genus $g\geqslant 3$. However, at this moment, it is unclear whether our methods can be potentially applied in such situations. For instance it is known that the mapping class group of a surface of genus above three do not surject onto any non-trivial free groups, since they are perfect groups \cite{Pow78}. Moreover, the same holds for many of their finite index subgroups (see \cite[Corollary 7.4]{HL95}, \cite[Theorem A]{M01}, and \cite[Theorem B]{Pu10}; see also \cite[Section 7]{Iv06}). However, we remark the possibility that these groups may admit natural surjections onto other type of groups to which the present deformation/rigidity theory methods could successfully apply.

\section{Proof of Theorem \ref{C}}\label{CC}

The goal of this section is to prove Theorem \ref{C}.
In the proof of Theorem \ref{C} we will combine the methods developed above with some powerful results  due to D. Osin \cite{Os05} and respectively F. Dahmani, G. Guirardel and D. Osin \cite{DGO11}. More precisely, we will need the following structural result for relatively hyperbolic groups:

\begin{cor}[Osin]\label{dgo-homo} Let $\Gamma$ be a group which is hyperbolic relative to a  finite family of proper finitely generated,  residually finite, infinite subgroups groups $\{H_\iota \,|\, \iota \in I \}$. 

Then there exist a non-elementary hyperbolic group $K$ and a surjective group homomorphism $\delta:\G \ra K$ such that $N=\ker(\delta)$ is isomorphic to a non-trivial (possible infinite) free product $N =\ast_{\iota\in I} \ast_{\gamma\in T_{\iota}} N_\iota^\gamma$, where $T_{\iota}\subset \Gamma$ are some subsets and $N^{\gamma} _\iota=\gamma N_\iota\gamma^{-1}$ is the conjugate of $N_\iota$ by $\gamma$. In particular, there exist infinite groups $L_1$ and $L_2$ such that $ker(\delta)=L_1\star L_2$.    \end{cor}

 We are grateful to Denis Osin for kindly communicating to us this result, which can be viewed as  a technical variation of \cite[Theorem 1.1]{Os05} and \cite[Theorem 7.9]{DGO11}.

\begin{proof}  By  combining \cite[Corollaries 1.7 and 4.5]{Os06} we can find two elements
  $\g,\la \in\Gamma$ of infinite order such that  $\Gamma $ is hyperbolic
 relative to the collection $\{H_\iota |\iota\in I\}\cup \{E_\G(\g), E_\G(\la)\}$. Here
 $E_\G(\g)$ and $E_\G(\la)$ denote the maximal, virtually cyclic subgroups containing $\g$ and $\la$, respectively. Notice that since the group $H_\iota$ is residually finite, for every $\iota\in I$, one can find a finite index normal subgroup $N_\iota\lhd H_\iota$ which avoids any given finite subset in $\Gamma \setminus \{e\}$. Thus,  any such a family of $N_\iota$'s  together with the trivial subgroups of $E_\G(\g)$ and $E_\G(\la)$ satisfies the hypothesis of  \cite[Theorem 1.1]{Os05}. This theorem implies that the quotient group $K:=\G /N$ by the normal closure $N:=\langle\cup_\iota N_\iota\rangle^\G$  of $\cup_\iota N_\iota$ in $\G$ is hyperbolic.  
 
 We claim that $K$ is non-elementary hyperbolic. To see this, let $\pi:\G\rightarrow K$ be the quotient homomorphism. Denote $A=\pi(E_\G(\g))$ and $B=\pi(E_\G(\la))$. Since by \cite[Theorem 1.1]{Os05}, the restriction of $\pi$ to  $ E_\G(\g)$ and  $E_\G(\la)$ is injective, we get that $A$ and $B$ are infinite, virtually cyclic subgroups of $K$. Since $E_\G(\g)\cap E_\G(\la)$ is finite by \cite[Theorem 1.7]{Os04}, we also get that $A\cap B$ is finite.  
 This shows that $K$ is not virtually cyclic, and hence is non-elementary hyperbolic.

 Moreover, by \cite[Corollary 6.34, Theorem 5.15]{DGO11}, the normal closure $N=\langle \cup_\iota N_\iota\rangle^\G$ is either trivial or it is isomorphic to a (possibly infinite) free product, $N =\ast_{\iota\in I} \ast_{\gamma\in T_{\iota}} N_\iota^\gamma$, where $T_{\iota}\subset \Gamma$ are some subsets and $N^{\gamma} _\iota$ is the conjugate of $N_\iota$ by $\gamma$.  Since  all the peripheral subgroups  $H_\iota$'s are almost malnormal in $\Gamma$ by \cite[Theorem 1.7]{Os04},  then $N$ is a non-trivial free product.  
 
 Letting $\delta:\G\ra K $ to be the canonical quotient map, where $N=ker(\delta)$, we get the desired conclusion.  \end{proof}

\begin{Remark}
If the parabolic groups $H_\iota$ are amenable then so are all the conjugates $N^{\gamma_\iota}_\iota$'s and we  conclude that the group  $N$ is weakly amenable, with Cowling-Haagerup constant one. Moreover, since $N$ is a free product of amenable groups,  it admits a proper $1$-cocycle into the  left regular representation.  Altogether, using Theorem \ref{pv}  and \cite[Theorem 1.4]{pv2} we have that  $N , \, \G /N\in \mathcal C_{rss}$, and hence $\G\in Quot_2(\mathcal C_{rss})$. Thus, by Theorem \ref{cartan} above it already follows that such groups $\Gamma$ are $\mathcal C$-rigid. 
However, Theorem \ref{C} asserts that the same holds without the amenability assumption on the parabolic groups. 
\end{Remark}

In order to treat the case when the parabolic groups $H_\iota$ are not necessarily amenable, we will use the following consequence of some recent structural results for normalizers inside amalgamated free product von Neumann algebras:

\begin{theorem}[Ioana, \cite{Io12a}; Vaes, \cite{Va13}]\label{freeprodcontrol}Let $L_1,L_2$ be two groups with $|L_1|\geqslant 2$ and $|L_2|\geqslant 3$. Let $L_1\star L_2 \ca (X,\mu)$ be a pmp action and denote  $M = L^\infty(X)\rtimes (L_1 \star L_2)$.  Let $p\in M$ be a nonzero projection and $A \subset pMp$ be an amenable von Neumann subalgebra.

If the normalizer $Q:=\mathcal N_{pMp}(A)''\subset pMp$ has finite index, then  $A \prec^s_M L^\infty(X)$. 

\end{theorem} 

Note that this result can be derived by adapting the proof \cite[Theorem 7.1]{Io12a}. However, for simplicity, we will give a more direct proof of Theorem \ref{freeprodcontrol}  using \cite{Va13}.

\begin{proof}  By applying \cite[Theorem A]{Va13} we get that one of the following statements holds:
\begin{enumerate}
\item  $Q$ is amenable relative to $L^\infty(X)$ inside $M$;
\item   $Q \prec_M L^\infty(X)\rtimes L_i$, for some $i\in\{1,2\}$;
\item $A \prec_M L^\infty(X)$.
\end{enumerate} 
Now, since $Q$ has finite index in $pMp$, then by Lemma \ref{ramen} we have that $pMp \prec^s_M Q$ and hence $pMp$ is amenable relative to $Q$ inside $M$. Thus, if  (1) holds,  then by Proposition 2.4 (3) in \cite{OP07} we would get that $pMp$ is amenable relative to $L^\infty(X)$ inside $M$. This would imply that $L_1\star L_2$ is amenable, which is a contradiction. 

Next, assume that (2) holds. Thus, we have that $Q \prec_M L^\infty(X)\rtimes L_i$, for some $i\in\{1,2\}$. By proceeding as in the beginning of Theorem \ref{Cartan} we can find projections $q_1, q_2\in \mathcal Z(Q)$ with $q_1+q_2=p$ and $q_1\neq 0$ such that  $Qq_1 \prec^s_M L^\infty(X)\rtimes L_i$ and $Qq_2 \nprec_M L^\infty(X)\rtimes L_i$. Since $Q \subseteq pMp$  has finite index it follows that $Qq_1\subseteq q_1Mq_1$ has finite index and hence by Lemma \ref{ramen} we get $q_1Mq_1  \prec^s_M Qq_1$. This fact in combination with the above and  Remark 3.7 in \cite{Va08} gives that $q_1Mq_1  \prec_M L^\infty(X)\rtimes L_i$. This further implies that $L_i$ has finite index in $L_1\star L_2$, which is again a contradiction. 

Thus, the only possibility is that $A \prec_M L^\infty(X)$. Finally, by arguing similarly to the beginning of Theorem \ref{Cartan}, we conclude that $A \prec^s_M L^\infty(X)$. 
\end{proof}

We are now ready to prove Theorem \ref{C}.

\subsection{Proof of Theorem \ref{C}}
 By Corollary \ref{dgo-homo}  there exists a surjective  homomorphism  $\delta:\G\ra \Lambda$  onto a  non-elementary hyperbolic group $\Lambda$ such that $ker(\delta)=L_1\star L_2$ is the free product of two infinite groups.  Let $\theta:M\rightarrow M\bar{\otimes}L(\Lambda)$ be the $*$-homomorphism associated to $\delta$ as defined in Notation \ref{comu}. 
We identify $M\bar{\otimes}L(\Lambda)=M\rtimes\Lambda$, where $\Lambda$ acts trivially on $M$. Since $\Lambda\in\mathcal C_{rss}$ by \cite{pv2} and $\theta(M)\subset\mathcal N_{(M\bar{\otimes}L(\Lambda))}(\theta(A))''$, we have that either $\theta(A)\prec_{M\bar{\otimes}L(\Lambda)}M\otimes 1$, or $\theta(M)$ is amenable relative to $M\otimes 1$ inside $M\bar{\otimes}L(\Lambda)$. 

By Proposition \ref{prop2}  the second case implies that $\delta(\G)=\Lambda$ is amenable, which is a contradiction. So the first case must hold , i.e., $\theta(A)\prec_{M\bar{\otimes}L(\Lambda)}M\otimes 1$. Then Proposition \ref{prop1} implies that 
$A\prec_{M}P:=L^\infty(X)\rtimes ker(\delta)$. By Proposition \ref{masa} we can find a non-zero projection $q\in P$ and a masa  $A_0\subset qPq$ such that 
$P_0:=\mathcal N_{qPq}(A_0)''\subset qPq$ has finite index. Moreover, we can find non-zero projections $p_0\in A$, $q'\in A_0'\cap qMq$, and a unitary $u\in M$ such that $u(Ap_0)u^*=A_0q'$. 

Now, since $ker(\delta)=L_1\star L_2$, we can decompose $P=(L^\infty(X)\rtimes L_1)\star_{L^\infty(X)}(L^\infty (X)\rtimes L_2)$ as an amalgamated free product. Thus, since $A_0\subset qPq$ is amenable and $P_0\subseteq qPq$ has finite index, by applying Theorem \ref{freeprodcontrol} we derive that $A_0\prec^s_{P} L^\infty(X)$. Moreover, since $u(Ap)u^*=A_0q'$, this further implies that $A\prec_M L^\infty(X)$. Finally,  by using  \cite[Theorem A.1]{Po01}, we can find a unitary $u\in M$ such that $A=uL^\infty (X)u^*$.  \hfill$\square$
\vskip 0.05in
\begin{Remark}
Theorem \ref{C} applies to many concrete families of groups $\Gamma$ that have been intensively studied over the last decades in various branches of topology, geometric group theory, or model theory -- for instance, whenever $\Gamma$ is in any of the following classes: 

\begin{enumerate}\item The fundamental group of any complete, finite-volume Riemannian manifold with pinched negative sectional curvature; these groups are hyperbolic relatively to their cusps groups which are finitely generated and nilpotent, and hence residually finite \cite{Bow99,Fa99}.

\item Any limit group, in the sense of Sela, over torsion free, relative hyperbolic groups with free abelian parabolics \cite{KM99,Se00}. Such groups are known to be hyperbolic relatively to a collection of free abelian subgroups from the work of D. Groves, \cite{Gr05}.
\end{enumerate}
\end{Remark}

By using Theorem \ref{C}  in combination with 
 \cite[Theorem A]{Io08} we obtain a new $W^*$-superrigidity result, in the spirit of similar results from \cite{CS11,CSU11}.  

\begin{cor}Let $\Gamma$ be a  residually finite, property (T) group which is hyperbolic relative to a  finite family of proper,  finitely generated,  infinite subgroups. Let $\Gamma\curvearrowright (X,\mu)$ be a free ergodic pmp profinite action and $\Lambda\curvearrowright (Y,\nu)$ be an arbitrary free ergodic pmp action.

If $L^{\infty}(X)\rtimes\Gamma\cong L^{\infty}(Y)\rtimes\Lambda$, then  the actions $\Gamma\curvearrowright X$, $\Lambda\curvearrowright Y$ are stably conjugate. 
\end{cor}

\end{document}